\documentclass[oneside]{amsart}
\usepackage[colorlinks=true,allcolors=blue!80!black]{hyperref}
\usepackage{tikz-cd}
\usetikzlibrary{matrix, snakes, patterns, shadows.blur, backgrounds}

\usepackage{amssymb, amsthm}
\usepackage{enumerate, amsfonts, amsxtra,amsmath,latexsym,epsfig, color}
\usepackage{mathrsfs, MnSymbol}
\usepackage{geometry}                		
\usepackage{graphicx}
\usepackage{subcaption}
\usepackage{autobreak}
\usepackage{color}
\usepackage{amsmath,amscd}	
\usepackage{mathtools}
\usepackage{stmaryrd}
\usepackage{animate}
\usepackage{tikz-cd}
\usepackage{soul}
\usepackage{faktor}
\usepackage{comment}

\usepackage{textgreek}  

\input xy
\xyoption{all}

\newtheorem {theorem}{Theorem}[section]
\newtheorem {lemma} [theorem] {Lemma}
\newtheorem {proposition} [theorem] {Proposition}
\newtheorem {corollary} [theorem] {Corollary}

\theoremstyle{definition}

\newtheorem{remark}[theorem]{Remark}

\DeclareMathOperator{\lcm}{lcm}
\DeclareMathOperator{\G}{G}

\DeclareMathOperator{\Ind}{Ind}
\DeclareMathOperator{\Gal}{Gal}

\title{Constructive Proof of Two Characterizations of Arithmetic Equivalence}
\author{Shaver Phagan}
\address{Department of Mathematics\\ Purdue University, West Lafayette, IN}
\email{shaverphagan@gmail.com}

\begin{document}
\maketitle
\section*{Abstract}
We prove a formula governing the combinatorics of cyclic group actions and extend a lemma of Hasse to account for ramification.  Consequently, we obtain the first constructive proofs of two theorems characterizing arithmetic equivalence.  We are also able to give new proofs of strong multiplicity one type theorems for Kronecker equivalence and weak Kronecker equivalence.
\section{Introduction}
A rational prime \(p\) admits a factorization \(p\mathcal{O}_K=\mathfrak{p}_1^{e_1}\cdots\mathfrak{p}_n^{e_n}\) into prime ideals in the ring of integers \(\mathcal{O}_K\) of a number field \(K\).  Letting \(f_i=[\mathcal{O}_K/\mathfrak{p}_i:\mathbb{Z}/p\mathbb{Z}]\) be the degree of the residual extension at \(\mathfrak{p}_i\), the multiset \(S_K(p)=\{f_1,...,f_n\}\) is known as the \textbf{splitting type} of \(p\) in \(K\).  Kronecker asked whether \(K\) is determined up to isomorphism by the splitting types of all rational primes.  In other words, does equality \(S_F(p)=S_K(p)\) for all \(p\) imply an isomorphism \(F\simeq K\)?  This can be thought of as a local-to-global principle for the arithmetic of Galois conjugates of primitive generators.\\
\indent Gassmann found nonisomorphic number fields \(F\) and \(K\) for which \(S_F(p)=S_K(p)\) for all unramified rational primes \(p\) \cite{gassmann_1926}.  However, his approach does not address the finitely many ramified primes and fundamentally relies on Hasse's Lemma describing the splitting types of unramified primes via Frobenius classes.  Around fifty years later, Perlis used the functional equation of the Dedekind zeta function to argue by contradiction that equality \(S_F(p)=S_K(p)\) for cofinitely many \(p\) implies equality \(S_F(p)=S_K(p)\) for all \(p\) \cite{perlis_equation_1977} (such a theorem is loosely referred to as a \textit{strong multiplicity one type theorem}).  Hence, Perlis' argument promoted Gassmann's examples to bonafide counterexamples to Kronecker's question: there are nonisomorphic number fields \(F\) and \(K\) for which \(S_F(p)=S_K(p)\) for all \(p\).  However, Perlis' argument does not indicate how one can compute the splitting type of a ramified prime as a function of the splitting types of unramified primes, though it guarantees the existence of such a function.  Similarly, Perlis-Stuart prove in \cite{stuart_new_1995} that \(S_F(p)=S_K(p)\) at every unramified \(p\) if and only if \(|S_K(p)|=|S_F(p)|\) for all unramified primes \(p\), but they acknowledge their methods provide no formula to compute the splitting type of a prime as a function of the splitting type cardinalities.\\
\indent In this short note, we extend Hasse's Lemma to account for ramification with Proposition \ref{staction}. This amounts to packaging well-known facts in algebraic number theory into the language of cyclic group actions.  We also prove a formula describing the combinatorics of cyclic group actions on finite sets.  This is Theorem \ref{formula}.  Bringing these results together to bear on the arithmetic equivalence problem, we are able to give the first constructive proofs of the theorems of Perlis and Perlis-Stuart.  Furthermore, using the extended Hasse Lemma, we are able to give new short proofs of strong multiplicity one type theorems for Kronecker equivalence and weak Kronecker equivalence.  For the reader's convenience, we the recall the definitions of the equivalence relations considered here.  Number fields \(F\) and \(K\) are:
\begin{enumerate} 
	\item\label{q} \textbf{arithmetically equivalent} if \(S_F(p)=S_K(p)\) for all unramified \(p\)
	\item\label{k} \textbf{Kronecker equivalent} if, for all unramified \(p\), \(1\in S_F(p)\) if and only if \(1\in S_K(p)\)
	\item\label{wk} \textbf{weakly Kronecker equivalent} if \(\gcd(S_{F}(p))=\gcd(S_{K}(p))\) for all unramified \(p\).
\end{enumerate}

\subsection*{Acknowledgements}
I would like to thank my advisor, Ben McReynolds, for helpful conversation and for his support.  I would also like to thank Anurag Sahay and Zachary Selk for helpful feedback. 

\section{Combinatorics of Cyclic Group Actions}\label{group}

The main theorem of this section--and hence the main theorem of \cite{stuart_new_1995}--can be derived from results in \cite{bourque_gcd_1992}.  However, to make the exposition more self-contained, we give a proof using only elementary methods.  Let \(C\) be a finite cyclic group of order \(n\), and let \(S\) be a finite (left) \(C\)-set. Given a divisor \(m\) of \(n\), let \(C^m=\{c^m|c\in C\}\), \(M_m=|C^m\backslash S|\), and for integers \(t,k\), define 
$$
f_k(t)=\prod\limits_{p|\gcd(k,t)\text{ prime}}p^{v_p(t)},
$$
where \(v_p\) is the \(p\)-adic valuation, so that
$$
v_p(f_k(t))=\begin{cases}
v_p(t), & p|k\text{ and }p|t\\
0, & \text{otherwise}.
\end{cases}
$$
\begin{lemma}\label{tr}
If \(T|m|n\), and \(\frac{m}{T}\) is square-free, and \(\Omega\) is the prime omega function counting prime divisors of a positive integer with multiplicity, then
\begin{equation}\label{treq}
\frac{\sum\limits_{T|d|m}(-1)^{\Omega(md)}M_{d}}{\prod\limits_{q|\frac{m}{T} \text{ prime}}(q-1)}=\sum\limits_{\frac{m}{T}f_{m/T}(T)|d|n}\gcd(T,d)a_{d},
\end{equation}
where \(a_d\) is the number of elements with fiber of cardinality \(d\) in the natural projection \(S\rightarrow C\backslash S\), i.e. \(a_d\) is the number of \(C\)-orbits in \(S\) with cardinality \(d\).
\end{lemma}
\begin{proof}
We induct on \(\Omega(m/T)\).  If \(\Omega(m/T)=0\), Equation \eqref{treq} reduces to \(M_m=\sum\limits_{d|n}\gcd(m,d)a_{d}\).  By definition \(M_n=|S|=\sum\limits_{d|n}da_{d}\), and if \(\omega\in S\) has \(|C\omega|=d\), then by orbit-stabilizer \(\text{Stab}_C(\omega)=C^{d}\).  For such \(\omega\), we know that \(|C^{m}\omega|=[C^m:C^{d}\cap C^m]=\frac{\lcm(m,d)}{m},\) so 
$$
M_{m}=\sum\limits_{d|n}\frac{md}{\lcm(m,d)}a_{d}=\sum\limits_{d|n}\gcd(m,d)a_{d},
$$
and the desired formula holds.  Suppose that \(\Omega(m/T)>0\), and observe that, since \(m/T\) is square-free, if \(p|\frac{m}{T}\) we have 
$$
\sum\limits_{T|d|m}(-1)^{\Omega(md)}M_{d}=\sum\limits_{pT|d|m}(-1)^{\Omega(md)}M_{d}+\sum\limits_{T|d|\frac{m}{p}}(-1)^{\Omega(md)}M_{d},
$$
so by induction
$$
\frac{\sum\limits_{T|d|m}(-1)^{\Omega(md)}M_{d}}{\prod\limits_{q|\frac{m}{pT} \text{ prime}}(q-1)}=\sum\limits_{\frac{m}{pT}f_{m/pT}(pT)|d|n}\gcd(pT,d)a_{d}-\sum\limits_{\frac{m}{pT}f_{m/pT}(T)|d|n}\gcd(T,d)a_{d}
$$
$$
=\sum\limits_{\frac{m}{pT}f_{m/pT}(T)|d|n}(\gcd(pT,d)-\gcd(T,d))a_{d}
$$
$$
=(p-1)\sum\limits_{\frac{m}{T}f_{m/T}(T)|d|n}\gcd(T,d)a_{d}.
$$
The second equation holds because \(f_{m/pT}(pT)=f_{m/pT}(T)\), and the last equation holds because \(\gcd(pT,d)=\gcd(T,d)\) if and only if \(v_p(d)\leq v_p(T)\), while \(\gcd(pT,d)=p\gcd(T,d)\) otherwise.
\end{proof}
Given \(T,m,n\) as in the lemma and such that \(\frac{m}{T}\) is maximal square-free (note this uniquely determines \(T\)), set 
\begin{equation}\label{N}
N_m=\frac{1}{T}\frac{\sum\limits_{T|d|m}(-1)^{\Omega(md)}M_{d}}{\prod\limits_{q|\frac{m}{T} \text{ prime}}(q-1)},
\end{equation}
and observe that
\begin{equation}\label{max}
N_m=\sum_{m|d|n}a_{d}.
\end{equation}

\noindent Indeed, any prime factor of \(T\) divides \(\frac{m}{T}\), by choice of \(T\).  Therefore, \(f_{m/T}(T)=T\).  Furthermore, since in this case \(m\) divides \(d\) in the right hand side of Equation \eqref{treq}, so does \(T\), so \(\gcd(T,d)=T\).  Hence, Equation \eqref{max} follows from Lemma \ref{tr}.  We thank Anurag Sahay for suggesting the following simplifying lemma and for providing its proof.

\begin{lemma}\label{clean}
The right hand side of Equation \eqref{N} is equal to
$$
\frac{1}{\varphi(m)}\sum\limits_{d|m}\mu(d)M_{\frac{m}{d}},
$$
where \(\mu\) is the M{\"o}bius function, and \(\varphi\) is Euler's totient.
\end{lemma}
\begin{proof}
Note that the denominator is just \(T\varphi(m/T)=\varphi(m)\), since \(m/T\) is the square-free kernel of \(m\).  Furthermore
$$\Omega(md)=\Omega(T^2m'd')=\Omega(m'd'),$$
where \(m'=m/T\) and \(d'=d/T\).  But then \(m'\) and \(d'\) are square-free, so \((-1)^{\Omega(m'd')}=\mu(m')\mu(d')\), and therefore
$$N_m=\frac{1}{\varphi(m)}\sum\limits_{T|d|m}\mu(m/T)\mu(d/T)M_d$$
$$=\frac{1}{\varphi(m)}\sum\limits_{d|\frac{m}{T}}\mu(m/T)\mu(d)M_{Td}$$
$$=\frac{1}{\varphi(m)}\sum\limits_{d|\frac{m}{T}}\mu\left(\frac{m}{dT}\right)M_{Td}=\frac{1}{\varphi(m)}\sum\limits_{d|\frac{m}{T}}\mu(d)M_{\frac{m}{d}},$$
where the second-to-last equation holds by multiplicativity of \(\mu\), since \(m/T\) is square-free and \(\mu(d)^2=1\).  The claim follows since the last sum is over the divisors of the square-free kernel of \(m\), which are precisely the support of \(\mu\) on the divisors of \(m\).
\end{proof}
\begin{theorem}\label{formula}
If \(m\) is a divisor of \(n\), then
$$
a_m=\sum\limits_{d|\frac{n}{m}}\frac{\mu(d)}{\varphi(md)}\sum\limits_{\delta|md}\mu(\delta)M_{\frac{md}{\delta}}.
$$
\end{theorem}
\begin{proof}
Let \(f(m)=N_{\frac{n}{m}}\) and \(g(m)=a_{\frac{n}{m}}\).  Equation \eqref{max} can be rewritten
$$
f(m)=\sum\limits_{d|m}g(d).
$$
Applying the M{\"o}bius inversion formula and Lemma \ref{clean} proves the claim.
\end{proof}

We now give a constructive proof of the Perlis-Stuart Theorem.  In the proof, we treat an \(\mathbb{N}\)-valued multiset \(f\) as a function \(f:\mathbb{N}\rightarrow \mathbb{N}\) and denote by \(|f|\) the quantity \(\sum_nf(n)\).  For instance \(f=\{1,2,2\}\) means \(f(1)=1,f(2)=2\), \(f(n)=0\) for \(n\neq1,2\), and \(|f|=3\).
\begin{corollary}[\cite{stuart_new_1995} Main Theorem]
Number fields \(F\) and \(K\) are arithmetically equivalent if and only if \(|S_p(F)|=|S_p(K)|\) for all unramified rational primes \(p\).
\end{corollary}
\begin{proof}
We need only prove the nontrivial direction, and for this it suffices to show that one can recover \(S_p(F)\) for any unramified \(p\) from knowledge of \(|S_p(F)|\) for all unramified \(p\).  By the Hasse Lemma \eqref{hasse}, we know that \(|S_p(F)|=|C_p\backslash G/G_F|\), where \(G\) is the group of a Galois extension containing \(F\), and \(C_p\) is the cyclic subgroup of \(G\) generated by a Frobenius element for \(p\).  But then if \(n=|C_p|\), for each divisor \(d\) of \(n\), Chebotarev density guarantees there is an unramified prime \(p_d\) such that \(C_{p_d}=C_p^d\), so by Theorem \ref{formula} and the Hasse Lemma, we have
$$
S_p(F)(m)=\sum\limits_{d|\frac{n}{m}}\frac{\mu(d)}{\varphi(md)}\sum\limits_{\delta|md}\mu(\delta)|S_{p_{md/\delta}}(F)|.
$$
\end{proof}
As a further application of Theorem \ref{formula}, we indulge in an efficient proof of a special property of the representation theory of cyclic groups, which amounts to injectivity of the map \(\Omega C\rightarrow\mathcal{R}_\mathbb{Q} C\) from the Burnside ring of \(C\) to the rational representation ring of \(C\).  This homomorphism has been heavily studied for finite groups, and a full classification of the kernel was recently accomplished in \cite{bartel_brauer_2015}.
\begin{corollary}
If \(C\) is a finite cyclic group and \(S_1\) and \(S_2\) are finite \(C\)-sets such that \(\mathbb{Q}S_1\simeq\mathbb{Q}S_2\) as \(C\)-representations, then \(S_1\) and \(S_2\) are isomorphic as \(C\)-sets.
\end{corollary}
\begin{proof}
The isomorphism \(\mathbb{Q}S_1\simeq\mathbb{Q}S_2\) implies \(|C^d\backslash S_1|=|C^d\backslash S_2|\) for any \(d\), for instance, by Burnside's Lemma.  It follows from Theorem \ref{formula} that \(S_1\) and \(S_2\) have the same orbit decomposition under the \(C\)-action, so they are isomorphic as \(C\)-sets, since \(C\) is cyclic.
\end{proof}

\section{The Extended Hasse Lemma and its Applications}\label{galois}
\subsection{Notation} Given a place \(\nu\) of a number field \(F\), we denote the residue field at \(\nu\) by \(\kappa(\nu)\). If \(K/F\) is a finite extension (not necessarily Galois) with \(\omega\) a place of \(K\) over \(\nu\), we denote the group of the cyclic extension \(\kappa(\omega)/\kappa(\nu)\) by \(\mathfrak{g}_{\omega/\nu}\). Recall that \(f\) is a residual degree of \(K\) over \(\nu\) if and only if there is \(\omega\) over \(\nu\) such that \(|\mathfrak{g}_{\omega/\nu}|=f\). If \(K/F\) is Galois with group \(G\), we identify \(\text{Gal}(K_\omega/F_\nu)\) with the decomposition group \(\mathcal{D}_{\omega/\nu}\subset G\). Recall that the inertia subgroup \(\mathcal{I}_{\omega/\nu}\) is defined as the kernel of the canonical epimorphism \(\mathcal{D}_{\omega/\nu}\rightarrow\mathfrak{g}_{\omega/\nu}\), and if \(F'\) is a subfield of \(K\) containing \(F\), with a place \(\eta\) divisible by \(\omega\), then \(\mathcal{D}_{\omega/\eta}=\mathcal{D}_{\omega/\nu}\cap G_{F'}\) and \(\mathcal{I}_{\omega/\eta}=\mathcal{I}_{\omega/\nu}\cap G_{F'}\), where \(G_{F'}=\Gal(K/F')\) (c.f. Proposition 22.a in Chapter I of \cite{serre_local_1979}).  To simplify notation, given a Galois extension \(K/\mathbb{Q}\) with group \(G\) and a place \(\omega\) of \(K\) over \(p\), we will write \(\mathcal{D}, \mathcal{I}, \mathfrak{g}\), respectively, instead of \(\mathcal{D}_{\omega/p}, \mathcal{I}_{\omega/p}, \mathfrak{g}_{\omega/p}\), so long as it won't lead to confusion. Furthermore, we arrange that \(\mathcal{D}=\mathcal{I}C\), where \(C\) is the cyclic group generated by some preimage \(c\) of a generator of \(\mathfrak{g}\) under the projection \(\mathcal{D}\rightarrow\mathfrak{g}\). Finally, for \(g\in\G\), we write \(g^G\) for the conjugacy class of \(g\) in \(G\).
\subsection{Extended Hasse Lemma}
When \(p\) is unramified, so that \(\mathcal{D}=C\), a well-known lemma of Hasse says the splitting type of \(p\) in a subfield \(F\) of \(K\) is given by 
\begin{equation}\label{hasse}S_F(p)=\left\{\frac{|Cg_iG_F|}{|G_F|}\right\}_{i=1}^n,\end{equation}
where \(g_1,...,g_n\in G\) are representatives for the double coset space \(C\backslash G/G_F\) (c.f. \cite{hasse_zahlbericht_1970} II Section 2.3).   The following proposition describes the splitting type of an arbitrary rational prime in terms of a cyclic group action and reduces to Hasse's Lemma at an unramified prime.  In the proof, we use the notation \(H^g=g^{-1}Hg\) for \(H\subset G, g\in G\).  

\begin{proposition}[Extended Hasse Lemma]\label{staction}
Let \(F\) be a number field contained in a finite Galois extension \(K/\mathbb{Q}\) with group \(G\).  If \(\mathcal{D}\) is a decomposition group fixing some place \(\omega\) of \(K\) over a rational prime \(p\), then \(S_p(F)\) is naturally identified with the multiset of cardinalities of fibers in the natural projection \(\mathcal{I}\backslash G/G_F\rightarrow\mathcal{D}\backslash G/G_F\).
\end{proposition}
\begin{proof}
The cyclic group \(C\) admits a natural left action on the coset space \(S=\mathcal{I}\backslash G/G_F\), whose orbit space is \(\mathcal{O}=\mathcal{D}\backslash G/G_F\).  The Orbit-Stabilizer Theorem, along with the well-known fact that a Galois group acts transitively on the prime divisors of a fixed prime in the base field, implies the orbits are in bijection with the set of places of \(F\) over \(p\).  Hence, we need only check that the fiber over \(\mathcal{D}gG_F\in\mathcal{O}\) in the projection \(S\rightarrow\mathcal{O}\) is of cardinality \(f_\nu\), where \(f_\nu=[\kappa(\nu):\mathbb{F}_p]\), and \(\nu\) is the place of \(F\) under \(g^{-1}\omega\).  This is again made transparent with the Orbit-Stabilizer Theorem.  Indeed, if \(C=<c>\) and \(\overline{g}:=\mathcal{I}gG_F\in\mathcal{O}\), we note that \(c^m\overline{g}=\overline{g}\) if and only if \(g^{-1}\mathcal{I}c^mg\cap G_F\neq\emptyset\), so that $$|C\overline{g}|=\frac{|C^g|}{|C^g\cap(\mathcal{I}^g(\mathcal{D}^g\cap G_F))|}=f_\nu,$$ since \(\mathcal{I}^g(\mathcal{D}^g\cap G_F)\) is the kernel of the composition \(\mathcal{D}^g\rightarrow\mathfrak{g}_{g^{-1}\omega/p}\rightarrow\mathfrak{g}_{\nu/p}\).
\end{proof}

\subsection{Applications}
We are now in a position to give a constructive proof Perlis' Theorem and prove the strong multiplicity one type theorems for Kronecker and weak Kronecker equivalence.  The proofs are broken into two parts: we study Frobenii of unramified primes to derive a conjugacy class condition from a splitting type property in (\textbf{Part 1}), then we use this class condition to determine the splitting type property for ramified primes in (\textbf{Part 2}).  Throughout, we take \(K\) to be a Galois extension of \(\mathbb{Q}\), containing \(K_1\) and \(K_2\).
\subsubsection{Arithmetic Equivalence (Perlis' Theorem)}

\begin{theorem}[\cite{perlis_equation_1977} Theorem 1.b]\label{qeq}
		If \(K_1\) and \(K_2\) are arithmetically equivalent, then \(S_{K_1}(p)=S_{K_2}(p)\) for any rational prime \(p\).
	\end{theorem}
	\begin{proof}
	(\textbf{Part 1}) Let \(F\) be a subfield of \(K\) and \(p\) an unramified rational prime.  If \(\mathcal{F}_p\) is the Frobenius class for \(p\) in \(G\), and \(z_p\) denotes the size \(|Z_G(c)|\) of the centralizer in \(G\) of an element \(c\) in \(\mathcal{F}_p\), it is a classical observation of Gassmann \cite{gassmann_1926} that 
	$$S_F(p)(1)|G_F|=|\mathcal{F}_p\cap G_F|z_p.$$  
	For a concise treatment, see Lemma 1 in \cite{perlis_equation_1977}.  In particular, given \(g\in G\), there is a prime \(p_g\) unramified in \(F\), whose Frobenius class is \(g^G\), and we may write
	\begin{equation}\label{almostconj}
	|g^G\cap G_F|=\frac{|G_F|}{z_{p_g}}S_F(p_g)(1).
	\end{equation}
		(\textbf{Part 2} ) Let \(q\) be a rational prime that ramifies in \(F\). Writing \(C\mathcal{I}=\mathcal{D}_{\omega/q}\), we identify \(S_F(q)\) with the multiset of orbit sizes in the natural action of \(C\) on \(\mathcal{I}\backslash G/G_F\), according to Proposition \ref{staction}.  If \(S_{m,q}=|C^m\mathcal{I}\backslash G/G_F|\), and \((V,\rho)\) is the representation induced by the right regular action of \(G\) on \(C^m\mathcal{I}\backslash G\), then
		\begin{equation}\label{fixeddim}
		S_{m,q}=\dim V^{G_F}=\frac{1}{|G_F|}\sum\limits_{g^G\in[G]}|g^G\cap G_F|\chi_{q,m}(g),
		\end{equation}  
		where \([G]\) is the set of conjugacy classes in \(G\), \(V^{G_F}\subset V\) is the subspace of \(G_F\)-invariant vectors, and \(\chi_{q,m}\) is the character of \(\rho\).  Combining Equations \eqref{almostconj} and \eqref{fixeddim}, we conclude
		\begin{equation}\label{almost}
		S_{m,q}=\sum\limits_{g^G\in [G]}\frac{S_F(p_g)(1)}{z_{p_g}}\chi_{q,m}(g).
		\end{equation}
		Observe now that \(S_{m,q}\) is equal to the number of orbits in the action of \(C^m\) on \(\mathcal{I}\backslash G/G_F\), so by Theorem \ref{formula} and Equation \eqref{almost}
		$$
		S_F(q)(m)=\sum\limits_{d|\frac{n}{m}}\frac{\mu(d)}{\varphi(md)}\sum\limits_{\delta|md}\mu(\delta)\sum\limits_{g^G\in[G]}\frac{S_F(p_g)(1)}{z_{p_g}}\chi_{q,md/\delta}(g).
		$$
		In particular, if \(S_{K_1}(p)=S_{K_2}(p)\) for every unramified prime \(p\), then \(S_{K_1}(p_g)=S_{K_2}(p_g)\) for each \(g\in G\), and therefore \(S_{K_1}(q)=S_{K_2}(q)\) for ramified primes \(q\).
	\end{proof}

\subsubsection{Kronecker and Weak Kronecker Equivalence}
Given a character \(\chi\) on \(G\), we exploit the fact that \(\chi\) is a class function and abuse notation by writing \(\chi(g^G):=\chi(g)\).  We similarly abuse the order function \(o\).  Recall that \(\Ind_H^G(1_H)(g)\neq 0\) if and only if \(g^G\cap H\neq\emptyset\), where \(1_H\) is the trivial representation for \(H\). The next lemma is useful and follows at once from Proposition \ref{staction}.

\begin{lemma}\label{charactersandresidues}
If \(m\) is a positive integer, there is \(g\in G\) such that $$g\mathcal{I}c^mg^{-1}\cap G_F\neq\emptyset$$ if and only if there is \(f\in S_F(p)\) such that \(f|m\).  In particular, if \(p\) is unramified in \(K\), then \(\chi_F(\mathcal{F}_p^m)\neq 0\) if and only if there is \(f\in S_F(p)\) such that \(f|m\), where \(\chi_F=\Ind_{G_{F}}^G(1_{G_F})\), and \(\mathcal{F}_p\subset G\) is the Frobenius class of \(p\).
\end{lemma}

\begin{theorem}[\cite{klingen_zahlkorper_1978} II, \cite{klingen_arithmetical_1998} III Theorem 2.9]\label{KrEq}
		If \(K_1\) and \(K_2\) are Kronecker equivalent, then for any rational prime \(p\), \(1\in S_{K_1}(p)\) if and only if \(1\in S_{K_2}(p)\).
\end{theorem}
\begin{proof}
(\textbf{Part 1}) For every \(g\in G\), there is a rational prime \(p\) unramified in \(K\) whose Frobenius class \(\mathcal{F}_p\) is precisely the conjugacy class of \(g\) in \(G\), so that \(1\in S_{F}(p)\) if and only if \(\chi_F(g)\neq 0\), by Lemma \ref{charactersandresidues}.  Since \(K_1\) and \(K_2\) are Kronecker equivalent, we conclude that \(\chi_{K_1}(g)\neq 0\) if and only if \(\chi_{K_2}(g)\neq 0\).  Therefore, 
	\begin{equation}\label{covering}
	\bigcup\limits_{g\in G_{K_1}}g^G=\bigcup\limits_{g\in G_{K_2}}g^G.
	\end{equation}
(\textbf{Part 2}) Suppose \(m\) is an integer divisible by some element of \(S_{K_1}(p)\), so there is a place \(\nu\) of \(K_1\) such that \(c^m\in\mathcal{I}\mathcal{D}_{\omega/\nu}\), by Lemma \ref{charactersandresidues}.  In particular, there are \(s\in\mathcal{I}\) and \(t\in\mathcal{D}_{\omega/\nu}\) such that \(c^m = st\).  Now, \(t\in G_{K_1}\), so there is \(g\in G\) such that \(g tg^{-1}\in G_{K_2}\), by Equation \eqref{covering}.  But then, if \(\omega'=g\omega\), we also know that \(g tg^{-1}\in\mathcal{D}_{\omega'/p}\), and \(g sg^{-1}\in\mathcal{I}_{\omega'/p}\). Hence \(g c^m g^{-1}\in \mathcal{I}_{\omega'/p}\mathcal{D}_{\omega'/\nu'}\), where \(\nu'\) is the place of \(K_2\) divisible by \(\omega'\).  Therefore, there is a divisor of \(m\) in \(S_{K_2}(p)\), again by Lemma \ref{charactersandresidues}.  By symmetry, we conclude that, for \(m\in\mathbb{Z}\), there is a divisor of \(m\) in \(S_{K_1}(p)\) if and only if there is a divisor of \(m\) in \(S_{K_2}(p)\).  For \(m=1\), this means \(1\in S_{K_1}(p)\) if and only if \(1\in S_{K_2}(p)\).
\end{proof}

We use an elementary lemma in the proof of the last theorem.
\begin{lemma}\label{SylowPullBack}
If \(\pi:G\rightarrow H\) is an epimorphism of finite groups, then a subgroup of \(H\) is Sylow \(p\) if and only if it is of the form \(\pi(G_p)\), where \(G_p\) is a Sylow \(p\)-subgroup of \(G\).
\end{lemma}

\begin{theorem}[\cite{lochter_weakly_1994}]
		If \(K_1\) and \(K_2\) are weakly Kronecker equivalent, then \(\gcd(S_{K_1}(p))=\gcd(S_{K_2}(p))\) for any rational prime \(p\).
	\end{theorem}
	\begin{proof}
	(\textbf{Part 1}) Let \(\delta_{i,p}=\gcd(S_{K_i}(p))\). If \(p\) is a rational prime unramified in \(K\), then \(\delta_{1,p}=\delta_{2,p}\).  Furthermore, if \(o(\mathcal{F}_p)\) is a \(q\)-power for some prime \(q\), then \(\delta_{i,p}\) is itself a residue over \(p\), so that \(\chi_{K_i}(\mathcal{F}_p^n)\neq 0\) if and only if \(\delta_{i,p}|n\), by Lemma \ref{charactersandresidues}.  In particular, \(\chi_{K_1}(\mathcal{F}_p)\neq 0\) if and only if \(\chi_{K_2}(\mathcal{F}_p)\neq 0\).  Hence, by Chebotarev density, if \(g\) is an element of a Sylow \(q\)-subgroup of \(G\), then \(G_{K_1}\) intersects the \(G\)-conjugacy class of \(g\) if and only if \(G_{K_2}\) does.  In other words, if \(G_{K_i,q}\) is a Sylow \(q\)-subgroup of \(G_{K_i}\), then
	\begin{equation}\label{qcovering}
	\bigcup\limits_{g\in G_{K_1,q}}g^G=\bigcup\limits_{g\in G_{K_2,q}}g^G.
	\end{equation}  
	(\textbf{Part 2}) Given a prime \(q\), Lemma \ref{charactersandresidues} tells us that \(\delta_{i,p}\) is not divisible by \(q^k\) if and only if there is a place \(\omega\) of \(K\) such that \(c^{n_qq^{k-1}}\in\mathcal{I}(\mathcal{D}\cap G_{K_i})\), where \(n_q=o(c)/q^{v_q(o(c))}\).  But then \(c^{n_qq^{k-1}}\) projects to an element of \(q\)-power order in \(\mathfrak{g}_{\omega/\nu}\), so by Lemma \ref{SylowPullBack}, we may write \(c^{n_qq^{k-1}}=\iota\gamma\) for some \(\iota\in\mathcal{I}\) and \(\gamma\) an element of a Sylow \(q\)-subgroup of \(\mathcal{D}_{\omega/\nu}\), where \(\nu\) is the place of \(K_i\) under \(\omega\).  We can then use Equation \eqref{qcovering}, along with the conjugation trick from the proof of Theorem \ref{KrEq} to deduce that \(q^k|\delta_{1,p}\) if and only if \(q^k|\delta_{2,p}\).  Hence, \(\delta_{1,p}\) and \(\delta_{2,p}\) have the same prime divisors, with the same multiplicity, so \(\delta_{1,p}=\delta_{2,p}\).
	\end{proof}
	
	\begin{remark}
	We conclude by remarking that the group-theoretic conditions derived in (\textbf{Part 1}) of the proofs above are classical (c.f. \cite{gassmann_1926, klingen_zahlkorper_1978, lochter_weakly_1994}).  The new insight is how the Extended Hasse Lemma allows this condition to be leveraged over the ramified primes in each case, in (\textbf{Part 2}).
	\end{remark}

\bibliography{ConstructiveProofs}

\begin{thebibliography}{Gas26}

\bibitem[BD15]{bartel_brauer_2015}
Alex Bartel and Tim Dokchitser.
\newblock Brauer relations in finite groups.
\newblock {\em Journal of the European Mathematical Society},
  17(10):2473--2512, October 2015.

\bibitem[BL92]{bourque_gcd_1992}
Keith Bourque and Steve Ligh.
\newblock On {GCD} and {LCM} matrices.
\newblock {\em Linear Algebra and its Applications}, 174:65--74, September
  1992.

\bibitem[Gas26]{gassmann_1926}
Fritz Gassmann.
\newblock Bemerkungen zur vorstehenden arbeit von hurwitz (Über beziehungen
  zwischen den primidealen eines algebraischen körpers und den substitutionen
  seiner gruppe).
\newblock {\em Mathematische Zeitschrift}, December 1926.

\bibitem[Has70]{hasse_zahlbericht_1970}
H.~Hasse.
\newblock {\em Zahlbericht}.
\newblock Physica-Verlag, 1970.

\bibitem[Kli78]{klingen_zahlkorper_1978}
Norbert Klingen.
\newblock Zahlkörper mit gleicher {Primzerlegung}.
\newblock {\em Journal für die reine und angewandte Mathematik},
  0299/0300:342--384, 1978.

\bibitem[Kli98]{klingen_arithmetical_1998}
Norbert Klingen.
\newblock {\em Arithmetical {Similarities}: {Prime} {Decomposition} and
  {Finite} {Group} {Theory}}.
\newblock Oxford {Mathematical} {Monographs}. Oxford University Press, Oxford,
  New York, April 1998.

\bibitem[Loc94]{lochter_weakly_1994}
Manfred Lochter.
\newblock Weakly {Kronecker} equivalent number fields.
\newblock {\em Acta Arithmetica}, 67(4):295--312, 1994.

\bibitem[Per77]{perlis_equation_1977}
Robert Perlis.
\newblock On the equation ζk(s) = ζk'(s).
\newblock {\em Journal of Number Theory}, 9(3):342--360, August 1977.

\bibitem[Ser79]{serre_local_1979}
Jean-Pierre Serre.
\newblock {\em Local {Fields}}, volume~67 of {\em Graduate {Texts} in
  {Mathematics}}.
\newblock Springer New York, New York, NY, 1979.

\bibitem[SP95]{stuart_new_1995}
D.~Stuart and R.~Perlis.
\newblock A {New} {Characterization} of {Arithmetic} {Equivalence}.
\newblock {\em Journal of Number Theory}, 53(2):300--308, August 1995.

\end{thebibliography}
\bibliographystyle{alpha}
\end{document}